\documentclass[reqno]{amsart}

\usepackage{amsmath}
\usepackage{amssymb}
\usepackage{amsthm}
\usepackage{mathrsfs}
\usepackage{mathtools}
\usepackage[hmarginratio=1:1]{geometry}
\usepackage{graphicx}
\usepackage{longtable}
\usepackage{booktabs}
\usepackage{float}
\usepackage{aliascnt}
\usepackage[font=footnotesize]{caption}
\usepackage{enumerate}
\usepackage{cleveref}
\usepackage{enumitem}

\newcommand{\ZZ}{\mathbb{Z}}
\newcommand{\QQ}{\mathbb{Q}}

\renewcommand{\AA}{\mathbb{A}}
\newcommand{\oo}{\mathfrak{o}}
\newcommand{\pp}{\mathfrak{p}}

\newcommand{\norm}{\text{N}}
\newcommand{\GL}{\text{GL}}

\renewcommand{\pmod}[1]{\text{ (mod }#1)}
\newcommand{\ord}{\text{ord}}
\newcommand{\pclass}{\text{cls}^{+}}
\newcommand{\pgenus}{\text{gen}^{+}}
\newcommand{\pmass}{m^{+}}

\makeatletter
\newcommand{\biggg}{\bBigg@\thr@@}
\newcommand{\Biggg}{\bBigg@{3.5}}
\makeatother

\theoremstyle{definition}
\newtheorem{theorem}{Theorem}[section]

\newtheorem{proposition}[theorem]{Proposition}
\newtheorem{lemma}[theorem]{Lemma}
\newtheorem{corollary}[theorem]{Corollary}
\newtheorem{remark}[theorem]{Remark}

\numberwithin{equation}{section}
\numberwithin{table}{section}

\allowdisplaybreaks[4]

\title{Class numbers of binary quadratic polynomials}

\author{Zichen Yang}
\address{Department of Mathematics\\The University of Hong Kong\\Pokfulam, Hong Kong}
\email{zichenyang.math@gmail.com}

\thanks{The research of the author was supported by the Dissertation Year Fellowship of the University of Hong Kong.}

\date{\today}

\begin{document}

\begin{abstract}
In this paper, we give a formula for the proper class number of a binary quadratic polynomial assuming that the conductor ideal is sufficiently divisible at dyadic places. This allows us to study the growth of the proper class numbers of totally positive binary quadratic polynomials. As an application, we prove finiteness results on totally positive binary quadratic polynomials with a fixed quadratic part and a fixed proper class number.
\end{abstract}

\maketitle

\section{Introduction}

Let $K$ be a number field. Denote by $\oo$ the ring of integers of $K$. Suppose that $Q(\mathbf{x})$ is a quadratic form with coefficients in $\oo$. One of the fundamental problems in the arithmetic theory of quadratic forms is to solve the Diophantine equation $Q(\mathbf{x})=n$ for $n\in\oo$. In geometric terms, it can be interpreted as the problem of finding a vector $v$ in a lattice $L$ in a quadratic space $(V,Q)$ over $K$ with $Q(v)=n$. Such a vector is called a representation of $n$ by $L$. 

One way to study the representations by $L$ is to consider representations by the localization $L_{\pp}$ for all places $\pp\in\Omega$, where $\Omega$ is the set of non-trivial places of $K$. It is a well-known fact that an element $n\in\oo$ is represented by $L_{\pp}$ for all places $\pp\in\Omega$ if and only if it is represented by a lattice in the proper genus of $L$. In a very special case that the proper genus of $L$ contains only one proper class, then an element $n\in\oo$ is represented by $L$ if and only if it is represented by $L_{\pp}$ for all places $\pp\in\Omega$, that is, the local-global principle holds for $L$. On the other hand, if the proper genus contains more than one proper class, then the number of proper classes, that is the so-called proper class number of $L$, can be viewed as a measurement of the obstruction of the local-global principle. So it is an important problem in the arithmetic theory of lattices to calculate the proper class number. One method to estimate the proper class number is to compute the proper mass. By the celebrated Siegel--Minkowski--Smith mass formula \cite{smith1867orders,minkowski1885untersuchungen,siegel1935uber,siegel1936uber,siegel1937uber}, the proper mass of $L$ factors into a product of local densities of $L$ over $K_{\pp}$ for all places $\pp\in\Omega$. The explicit formulas for the local densities are known in many cases \cite{kitaoka1999arithmetic,gan2000group,cho2015group}.

It is natural to investigate extensions of this theory to inhomogeneous quadratic polynomials. Analogous to the connection between quadratic forms and lattices, for a quadratic polynomial, one can associate it with a shifted lattice $X$ in a quadratic space $V$ over $K$. To be precise, a shifted lattice is a subset of $V$ of the form $L+\nu$ with a lattice $L$ in $V$ and a vector $\nu\in V$. The notions of localization, the proper class number, and the proper mass can be defined for shifted lattices in a similar manner. For a systematic treatment, see \cite[Section 4]{chan2012representations}. Siegel \cite{siegel1935uber,siegel1936uber,siegel1937uber} himself studied the mass formula for shifted lattices. But he excluded totally positive cases, which were later covered by Van der Blij \cite{van1949theory}. Recently this topic was picked up by Sun \cite{sun2016class,sun2018growth}. In \cite[Theorem 2.4]{sun2016class}, Sun proved a mass formula for shifted lattices over number fields. Comparing to the mass formula for lattices, we have
\begin{equation}
\label{eqn::vanderblij}
\pmass(X)=\pmass(L)\prod_{\pp\mid\mathfrak{M}_X}[O^{+}(L_{\pp}):O^{+}(X_{\pp})],
\end{equation}
where $\mathfrak{M}_X$ is the conductor ideal of $X$ that is generated by the elements $a\in\oo$ such that $a\nu\in L$ and we denote by $\pmass(X)$ and $O^{+}(X)$ the proper mass and the proper orthogonal group of a shifted lattice $X$, respectively. 

In this paper, we are interested in the proper class number of a shifted lattice in a quadratic space $V$ of dimension $2$. In this case, the adelization $O_{\AA}^{+}(V)$ is commutative. Thus (\ref{eqn::vanderblij}) implies that
$$
h^{+}(X)=\frac{h^{+}(L)}{[O^{+}(L):O^{+}(X)]}\prod_{\pp\mid\mathfrak{M}_X}[O^{+}(L_{\pp}):O^{+}(X_{\pp})],
$$
where we denote by $h^{+}(X)$ the proper class number of a shifted lattice $X$. Therefore, to obtain a class number formula, it suffices to calculate the group index $[O^{+}(L_{\pp}):O^{+}(X_{\pp})]$ for finitely many places $\pp\in\Omega$. In \cite[Lemmas 3.3--3.5]{sun2018growth}, Sun calculated the group index $[O^{+}(L_{\pp}):O^{+}(X_{\pp})]$ when $\pp\nmid2$. However, the calculations contain errors. Thus one goal of this paper is to prove a corrected formula when $\pp\nmid2$. Furthermore, by the same method, we obtain a simple and exact formula for the group index when $\pp\mid2$ under a mild divisibility condition on the conductor ideal.

To perform the local computation, we need to fix a local structure of the shifted lattice $X=L+\nu$. To be precise, we write $L\simeq G$ if the Gram matrix of the lattice $L$ with respect to some basis of $L$ is $G$. Let $\Omega_f$ denote the set of all non-archimedean places of $K$. For any place $\pp\in\Omega_f$, we show in Lemma \ref{thm::threetypes} that if $\mathfrak{s}_{L_{\pp}}\subseteq\oo_{\pp}$, then there exists a basis of $L_{\pp}$ such that $L_{\pp}$ is of one of the following types
\begin{enumerate}
\item $L_{\pp}\simeq cD(1,\alpha)$ with $c,\alpha\in\oo_{\pp}$,
\item $L_{\pp}\simeq cA(0,0)$ with $c\in\oo_{\pp}$,
\item $L_{\pp}\simeq cA(\alpha,\beta)$ with $c,\alpha,\beta\in\oo_{\pp}$ such that $1\leq\ord_{\pp}(\alpha)\leq\min(\ord_{\pp}(2),\ord_{\pp}(\beta))$,
\end{enumerate}
where the matrices $D(\alpha,\beta)$ and $A(\alpha,\beta)$ are defined in Section \ref{sec::startlocal}. Moreover, if $K_{\pp}$ is non-dyadic, then there exists a basis such that $L_{\pp}$ is of type (1). We say that a collection $\{M_{\pp}\}_{\pp\mid\mathfrak{M}_X}$ consisting of matrices of the above three types is a local structure of $X$, if $L_{\pp}\simeq M_{\pp}$ for all places $\pp\mid\mathfrak{M}_X$ and $M_{\pp}$ is of type (1) for all places $\pp\mid\mathfrak{M}_X$ such that $\pp\nmid2$.

Since the class numbers are invariant under scaling, we may assume that $\mathfrak{s}_L\subseteq\oo$. For a shifted lattice $X=L+\nu$ in a quadratic space of dimension $2$ over $K$, we define
$$
\beta_2^{+}(X)\coloneqq\prod_{\pp\mid\gcd(\mathfrak{M}_X,2\oo)}\beta_{\pp}^{+}(L_{\pp};L_{\pp}),
$$
where $\beta_{\pp}^{+}(L_{\pp};L_{\pp})$ is defined in (\ref{eqn::definebeta}). Let $\norm_{K/\QQ}$ denote the norm of the extension $K/\QQ$. 

\begin{theorem}
\label{thm::maintheorem1}
Let $X=L+\nu$ be a shifted lattice in a quadratic space $(V,Q)$ of dimension $2$ over $K$ such that $\det(V)\not\in-(K^{\times})^2$. Assume that $\mathfrak{s}_L\subseteq\oo$. Fix a local structure $\{M_{\pp}\}_{\pp\mid\mathfrak{M}_X}$ of $X$ such that the following assumptions hold
\begin{enumerate}
\item $\ord_{\pp}(\mathfrak{M}_X)\geq\ord_{\pp}(2\alpha)+1$, if $M_{\pp}=cD(1,\alpha)$ for $\pp\mid\gcd(\mathfrak{M}_X,2\oo)$,
\item $\ord_{\pp}(\mathfrak{M}_X)\geq\ord_{\pp}(2)+1$, if $M_{\pp}$ is of type (3) for $\pp\mid\gcd(\mathfrak{M}_X,2\oo)$.
\end{enumerate}
Then we have
$$
h^{+}(X)=\frac{h^{+}(L)}{[O^{+}(L):O^{+}(X)]}\cdot\frac{\beta_2^{+}(X)\norm_{K/\QQ}(\mathfrak{M}_X)}{\norm_{K/\QQ}(\gcd(\mathfrak{M}_X,\mathfrak{I}_X))}\prod_{\substack{\pp\mid\mathfrak{M}_X,\pp\nmid2\\M_{\pp}=cD(1,\alpha)}}\left(1-\frac{\eta(-\alpha)}{\norm_{K/\QQ}(\pp)}\right),
$$
where $\eta(x)$ is defined in (\ref{eqn::defineeta}) and $\mathfrak{I}_X$ is an ideal of $\oo$ defined by
$$
\mathfrak{I}_X\coloneqq\prod_{\substack{\pp\mid\mathfrak{M}_X\\M_{\pp}=cD(1,\alpha)}}\pp^{\min(\ord_{\pp}(s_1),\ord_{\pp}(s_1^2+\alpha s_2^2))},
$$
with $L_{\pp}\simeq M_{\pp}=cD(1,\alpha)$ with respect to a basis $\{e_1,e_2\}$ and
$$
\nu=\frac{s_1e_1+s_2e_2}{\pi_{\pp}^{\mathfrak{t}_m}}
$$
for $s_1,s_2\in\oo_{\pp}$, $\mathfrak{t}_m=\ord_{\pp}(\mathfrak{M}_X)$, and a uniformizer $\pi_{\pp}$ of $K_{\pp}$.
\end{theorem}

As an application, we study the growth of the proper class numbers of totally positive shifted lattices as the norm of the conductor ideal increases.

\begin{corollary}
\label{thm::maintheorem2}
Suppose that $K$ is a totally real number field. Let $X=L+\nu$ be a totally positive shifted lattice in a quadratic space of dimension $2$ over $K$. For any positive number $\varepsilon>0$, we have
$$
h^{+}(X)\gg_{L,\varepsilon}\norm_{K/\QQ}(\mathfrak{M}_X)^{1-\varepsilon},
$$
as $\norm_{K/\QQ}(\mathfrak{M}_X)\to\infty$. The implied constant is effectively computable.
\end{corollary}

Therefore, fixing the totally positive lattice part, there are at most finitely many shifted lattices with a fixed proper class number.

\begin{corollary}
\label{thm::maintheorem3}
Suppose that $K$ is a totally real number field. Let $L$ be a totally positive lattice in a quadratic space of dimension $2$ over $K$. For a fixed integer $h\geq1$, there are only finitely many shifted lattices of the form $X=L+\nu$ such that $h^{+}(X)=h$.
\end{corollary}

The paper is organized as follows. In Section \ref{sec::pre}, we introduce basic notions about shifted lattices and prove some preparatory results for local computations. In Section \ref{sec::localcomputation}, we compute the group index $[O^{+}(L_{\pp})\colon O^{+}(X_{\pp})]$. Finally, in Section \ref{sec::main}, we combine the local computations to prove the main results.

\section{Preliminaries}
\label{sec::pre}

\subsection{Class numbers of shifted lattices}

Let $K$ be either a number field or a local completion of a number field. We denote by $\oo$ the ring of integers of $K$. Note that $\oo=K$ when $K$ is an archimedean completion. A quadratic space of dimension $r$ over $K$ is a vector space $V$ of dimension $r$ over $K$, equipped with a non-degenerate symmetric bilinear form $B\colon V\times V\to K$. The quadratic map $Q$ associated with $V$ is defined by $Q(v)\coloneqq B(v,v)$ for $v\in V$. A lattice in a quadratic space $V$ over $K$ is a finitely generated $\oo$-submodule of $V$, which spans $V$ over $K$. A shifted lattice in a quadratic space $V$ over $K$ is a subset of $V$ of the form $L+\nu$ with a lattice $L$ in $V$ and a vector $\nu\in V$. The conductor ideal $\mathfrak{M}_X$ of a shifted lattice $X=L+\nu$ is the integral ideal of $\oo$ consisting of the elements $a\in\oo$ such that $a\nu\in L$.

The proper orthogonal group $O^{+}(V)$ of $V$ is defined by
$$
O^{+}(V)\coloneqq\{\sigma\in\GL(V)\mid B(\sigma(v_1),\sigma(v_2))=B(v_1,v_2)\text{ for }v_1,v_2\in V\text{ and }\det(\sigma)=1\},
$$
which acts on the set of shifted lattices in $V$ by sending $X$ to $\sigma(X)$. The proper orthogonal group $O^{+}(X)$ of a shifted lattice $X$ is the stabilizer of $X$ under the action of $O^{+}(V)$. The proper class of a shifted lattice $X$, denoted by $\pclass(X)$, is the orbit of $X$ under the action of $O^{+}(V)$.

Now assume that $K$ is a number field. Denote by $\Omega$ the set of all non-trivial places of $K$ and denote by $\Omega_f$ the set of all non-archimedean places of $K$. For a quadratic space $V$ over $K$, the localization $V_{\pp}$ of $V$ at a place $\pp\in\Omega$ is the quadratic space $V\otimes K_{\pp}$ over $K_{\pp}$ with the bilinear form $B_{\pp}$ uniquely determined by the property $B_{\pp}(v_1\otimes1,v_2\otimes1)=B(v_1,v_2)$. Clearly we can identify $V$ as a $K$-subspace of $V_{\pp}$. For a shifted lattice $X$ in $V$, the localization $X_{\pp}$ is the shifted lattice in $V_{\pp}$ given by $L_{\pp}+\nu$, where $L_{\pp}$ is the $\oo_{\pp}$-submodule of $V_{\pp}$ generated by $L$. If $K$ is a totally real number field, a shifted lattice $X$ in a quadratic space $V$ is totally positive if the bilinear form $B_{\pp}$ is positive definite for any archimedean place $\pp\in\Omega$.

For a shifted lattice $X$ in $V$, the proper genus of $X$, denoted by $\pgenus(X)$, is defined to be the set of shifted lattices $Y$ in $V$ such that $Y_{\pp}\in\pclass(X_{\pp})$ for any place $\pp\in\Omega$. By \cite[Corollary 4.4]{chan2012representations}, the number of proper classes in the proper genus of $X$ is finite, which is called the proper class number of $X$, denoted by $h^{+}(X)$. Moreover, an element $n\in K$ is represented by a shifted lattice $Y\in\pgenus(X)$ if and only if $n$ is represented by $X_{\pp}$ at all places $\pp\in\Omega$ by \cite[Corollary 4.8]{chan2012representations}. Thus, the proper class number measures the obstruction of the local-global principle for a shifted lattice. 

To estimate the proper class number, we study the proper mass. From now on, we assume that
\begin{equation}
\label{eqn::assumption}
\dim(V)\geq3,\text{ or }\dim(V)=2\text{ and }\det(V)\not\in-(K^{\times})^2.
\end{equation}
Under the assumption (\ref{eqn::assumption}), we fix a normalized Tamagawa measure
\begin{equation}
\label{eqn::definemass}
\mu=\mu_{\infty}\prod_{\pp\in\Omega_f}\mu_{\pp}
\end{equation}
on the adelization $O_{\AA}^{+}(V)$. See \cite{tamagawa1966adeles} for a detailed definition. Let $O^{+}(V_{\infty})$ denote the archimedean component of $O_{\AA}^{+}(V)$. Then the proper mass of a shifted lattice $X$ is defined by
$$
\pmass(X)\coloneqq\sum_{\pclass(Y)\subseteq\pgenus(X)}\mu_{\infty}(O^{+}(V_{\infty})/O^{+}(Y)),
$$
where the sum runs over all proper classes in $\pgenus(X)$.

\begin{proposition}
\label{thm::massformula}
Let $V$ be a quadratic space over $K$ satisfying the assumption (\ref{eqn::assumption}). For any shifted lattice $X=L+\nu$ in $V$, we have
\begin{equation}
\label{eqn::massformula}
\pmass(X)=\pmass(L)\prod_{\pp\mid\mathfrak{M}_X}[O^{+}(L_{\pp}):O^{+}(X_{\pp})].
\end{equation}
\end{proposition}
\begin{proof}
Since the measure $\mu_{\pp}$ is invariant under translations, by using \cite[Theorem 2.4]{sun2016class}, we have
\begin{align*}
\pmass(X)=&~2\prod_{\pp\in\Omega_f}\mu_{\pp}(O^{+}(X_{\pp}))^{-1}=2\prod_{\pp\in\Omega_f}\mu_{\pp}(O^{+}(L_{\pp}))^{-1}[O^{+}(L_{\pp}):O^{+}(X_{\pp})]\\
=&~2\prod_{\pp\in\Omega_f}\mu_{\pp}(O^{+}(L_{\pp}))^{-1}\prod_{\pp\in\Omega_f}[O^{+}(L_{\pp}):O^{+}(X_{\pp})]=\pmass(L)\prod_{\pp\in\Omega_f}[O^{+}(L_{\pp}):O^{+}(X_{\pp})].
\end{align*}
If $\pp\nmid\mathfrak{M}_X$, there exists $a\in\oo_{\pp}^{\times}$ such that $a\nu\in L_{\pp}$ and thus we have $X_{\pp}=L_{\pp}$. So we have
$$
\pmass(X)=\pmass(L)\prod_{\pp\in\Omega_f}[O^{+}(L_{\pp}):O^{+}(X_{\pp})]=\pmass(L)\prod_{\pp\mid\mathfrak{M}_X}[O^{+}(L_{\pp}):O^{+}(X_{\pp})],
$$
as desired.
\end{proof}
\begin{remark}
If $X$ is totally positive, then the orthogonal group $O^{+}(Y)$ is finite for any shifted lattice $Y\in\pgenus(X)$. Therefore, one may alternatively define the proper mass of $X$ by
$$
\pmass(X)\coloneqq\sum_{\pclass(Y)\subseteq\pgenus(X)}|O^{+}(Y)|^{-1},
$$
which differs from the one in (\ref{eqn::definemass}) by a multiplication of $\mu_{\infty}(O^{+}(V_{\infty}))$. For this alternative definition, the formula (\ref{eqn::massformula}) holds as well.
\end{remark}

If $V$ is of dimension $2$, then the adelization $O_{\AA}^{+}(V)$ is commutative. Thus the mass formula gives a class number formula.

\begin{proposition}
\label{thm::classnumberformula}
Let $V$ be a quadratic space of dimension $2$ over $K$ such that $\det(V)\not\in-(K^{\times})^2$. For a shifted lattice $X=L+\nu$ in $V$, we have
$$
h^{+}(X)=\frac{h^{+}(L)}{[O^{+}(L):O^{+}(X)]}\prod_{\pp\mid\mathfrak{M}_X}[O^{+}(L_{\pp}):O^{+}(X_{\pp})].
$$
\end{proposition}
\begin{proof}
By \cite[\S 43C]{o2013introduction}, the adelization $O_{\AA}^{+}(V)$ is commutative. It follows that $O^{+}(Y)=O^{+}(X)$ for any shifted lattice $Y\in\pgenus(X)$. Thus, we have
$$
m^{+}(X)=\mu_{\infty}(O^{+}(V_{\infty})/O^{+}(X))h^{+}(X),
$$
and in particular
$$
m^{+}(L)=\mu_{\infty}(O^{+}(V_{\infty})/O^{+}(L))h^{+}(L).
$$
Since $O^{+}(X)$ is a subgroup of $O^{+}(L)$ of finite index, by using Proposition \ref{thm::massformula}, we have
$$
h^{+}(X)=\frac{h^{+}(L)}{[O^{+}(L):O^{+}(X)]}\prod_{\pp\mid\mathfrak{M}_X}[O^{+}(L_{\pp}):O^{+}(X_{\pp})],
$$
as desired.
\end{proof}

\subsection{Jordan canonical forms}
\label{sec::startlocal}

The discussion from now on is purely local until Section \ref{sec::main}. In order to ease the notations, we assume that $K$ is a finite extension over the $p$-adic field $\QQ_p$ for a prime number $p$. Denote by $\oo$, $\pp$, and $\pi$ the ring of integers, the unique prime ideal, and a uniformizer of $K$, respectively. The $\pp$-adic valuation $\ord_{\pp}\colon K\to\ZZ\cup\{\infty\}$ is normalized so that $\ord_{\pp}(\pi)=1$. Note that $\ord_{\pp}(0)\coloneqq\infty$ by convention. Furthermore, we set $e\coloneqq\ord_{\pp}(2)$. 

Take a lattice $L$ in a quadratic space of dimension $2$ over $K$. Let $B$ denote the bilinear form. Since $\oo$ is a principal ideal domain, the lattice $L$ is a free $\oo$-module of rank $2$. With respect to a basis $\{e_1,e_2\}$ of $L$, the Gram matrix $G$ of $L$ is the matrix given by
$$
G\coloneqq
\begin{pmatrix}
B(e_1,e_1) & B(e_1,e_2)\\
B(e_2,e_1) & B(e_2,e_2)
\end{pmatrix}.
$$
For a matrix $M\in\text{Mat}_{2\times2}(K)$, we write $L\simeq M$ if the Gram matrix of $L$ with respect to a basis is given by $M$. For $\alpha,\beta\in K$, we denote by $D(\alpha,\beta)$ the diagonal matrix with entries $\alpha,\beta$ and denote by
$$
A(\alpha,\beta)\coloneqq
\begin{pmatrix}
\alpha & 1\\
1 & \beta
\end{pmatrix}.
$$
In order to simplify the calculations of the group indices, we choose a suitable basis of $L$ so that the Gram matrix is as simple as possible. This is achieved in the following lemma. Note that the group index $[O^{+}(L)\colon O^{+}(X)]$ is invariant under scaling. So we may assume that $\mathfrak{s}_L\subseteq\oo$.

\begin{lemma}
\label{thm::threetypes}
Let $L$ be a lattice in a quadratic space of dimension $2$ over $K$ such that $\mathfrak{s}_L\subseteq\oo$. If $\pp\nmid2$, then we have $L\simeq cD(1,\alpha)$ with $c,\alpha\in\oo$. If $\pp\mid2$, then there exists a basis of $L$ such that $L$ is of one of the following types
\begin{enumerate}
\item $L\simeq cD(1,\alpha)$ with $c,\alpha\in\oo$,
\item $L\simeq cA(0,0)$ with $c\in\oo$,
\item $L\simeq cA(\alpha,\beta)$ with $c,\alpha,\beta\in\oo$ such that $1\leq\ord_{\pp}(\alpha)\leq\min(e,\ord_{\pp}(\beta))$.
\end{enumerate}
\end{lemma}
\begin{proof}
By \cite[\S 91C]{o2013introduction}, we have either $L=L_1\perp L_2$ with $L_1$ and $L_2$ modular lattices of rank $1$, or the lattice $L$ itself is modular. In the first case, we already have a basis with respect to which the Gram matrix is of the form $cD(1,\alpha)$ with $c,\alpha\in\oo$. So we assume that the second case holds henceforth. By a scaling, we may further assume that $L$ is unimodular. If $\pp\nmid2$, then we have $L\simeq D(1,\alpha)$ with $\alpha\in\oo^{\times}$ by \cite[92:1]{o2013introduction}. If $\pp\mid2$, then we have $L\simeq A(\alpha,\beta)$ with $\alpha,\beta\in\oo$ by using \cite[93:10]{o2013introduction}. Moreover, we may assume that $\ord_{\pp}(\alpha)\leq\ord_{\pp}(\beta)$ by symmetry. This case splits into four subcases. Firstly, if $\ord_{\pp}(\alpha)=\infty$, then $L\simeq A(0,0)$. Secondly, if $\ord_{\pp}(\alpha)=0$, then $L\simeq cD(1,\gamma)$ for $c,\gamma\in\oo$ by eliminating off-diagonal entries. Thirdly, if $\ord_{\pp}(\alpha)\geq e$, then either $L\simeq A(0,0)$ or $L\simeq A(2,2\rho)$ with $\rho\in\oo^{\times}$ by \cite[93:11]{o2013introduction}, which is of the third type. For the remaining cases, it is clear that $L$ is of the third type as well. This finishes the proof.
\end{proof}

\subsection{Some lemmas}

In this section, we keep using the local setting given in Section \ref{sec::startlocal}. To calculate the group indices, we have to prove some lemmas counting solutions to certain systems of congruences. For $x\in\oo$, we set
\begin{equation}
\label{eqn::defineeta}
\eta(x)\coloneqq
\begin{dcases}
1,&\text{ if }x\in\oo^{\times},x\equiv y^2\pmod{\pp}\text{ for some }y\in\oo^{\times},\\
-1,&\text{ if }x\in\oo^{\times},x\not\equiv y^2\pmod{\pp}\text{ for any }y\in\oo^{\times},\\
-\norm(\pp),&\text{ if }x\in\pp,
\end{dcases}
\end{equation}
where we denote by $\norm=\norm_{K/\QQ_p}$ the norm of the extension $K/\QQ_p$, and therefore we have $N(\pp)=p^f$ with $f$ being the inertial degree of the extension $K/\QQ_p$.

\begin{lemma}
\label{thm::system0}
Assume that $\pp\nmid2$. For $\alpha\in\oo$ and $t\geq1$, the number of solutions $(x_1,x_2)$ modulo $\pp^t$ to the congruence 
\begin{equation}
\label{eqn::system0}
x_1^2+\alpha x_2^2\equiv1\pmod{\pp^t}
\end{equation}
is $\norm(\pp)^t(1-\eta(-\alpha)\norm(\pp)^{-1})$.
\end{lemma}
\begin{proof}
The number of solutions modulo $\pp$ to the congruence 
$$
x_1^2+\alpha x_2^2\equiv1\pmod{\pp}
$$
is $\norm(\pp)(1-\eta(-\alpha)\norm(\pp)^{-1})$. In fact, it follows from \cite[Lemma 1.3.2]{kitaoka1999arithmetic} if $\alpha\in\oo^{\times}$. If $\alpha\in\pp$, then it is obvious. Since each solution modulo $\pp$ lifts to $\norm(\pp)^{t-1}$ solution modulo $\pp^t$ by Hensel's lemma, we see that the number of solutions to the congruence (\ref{eqn::system0}) is $\norm(\pp)^t(1-\eta(-\alpha)\norm(\pp)^{-1})$.
\end{proof}

The following lemma is in essence a consequence of Hensel's lemma. We reformulate it to unify the discussions of non-dyadic cases and dyadic cases.

\begin{lemma}
\label{thm::system1}
Suppose that $\alpha,s\in\oo$. We set $\mathfrak{t}_s\coloneqq\ord_{\pp}(s)$ and set $\mathfrak{t}_{\nu}\coloneqq\ord_{\pp}(s^2+\alpha)$. Assume that 
\begin{enumerate}
\item $\mathfrak{t}_m\geq e+1$;
\item $\mathfrak{t}_m+\mathfrak{t}_{\nu}-2\min(\mathfrak{t}_m,\mathfrak{t}_{\nu})\geq e$ if $\mathfrak{t}_s\geq\mathfrak{t}_{\nu}$;
\item $\mathfrak{t}_m+\mathfrak{t}_{\nu}-\min(\mathfrak{t}_m,\mathfrak{t}_{\nu})\geq\mathfrak{t}_s+e+1$ if $\mathfrak{t}_s<\mathfrak{t}_{\nu}$.
\end{enumerate}
For $t\geq\mathfrak{t}_m$, the number of solutions $(x_1,x_2)$ modulo $\pp^t$ to the system of congruences given by
\begin{equation}
\label{eqn::system1}
\begin{dcases}
x_1^2+\alpha x_2^2\equiv1\pmod{\pp^{t+e}},\\
x_1\equiv1-sx_2\pmod{\pp^{\mathfrak{t}_m}},\\
x_2\equiv0\pmod{\pp^{\mathfrak{t}_m-\mathfrak{t}_{\nu}}},
\end{dcases}
\end{equation}
is $\norm(\pp)^{t-\mathfrak{t}_m+\min(\mathfrak{t}_m,\mathfrak{t}_{\nu},\mathfrak{t}_s)}$.
\end{lemma}
\begin{proof}
Suppose that $(x_1,x_2)$ is a solution to (\ref{eqn::system1}). Using the assumption (1), we see that
$$
x_1^2\equiv 1-2sx_2+s^2x_2^2\pmod{\pp^{\mathfrak{t}_m+e}}.
$$
Combining with the first congruence in (\ref{eqn::system1}), we see that
\begin{equation}
\label{eqn::preprerelation}
2sx_2-(s^2+\alpha)x_2^2\equiv0\pmod{\pp^{\mathfrak{t}_m+e}}.
\end{equation}

Now we want to show that the system is unsolvable if $x_2\not\in\pp^{\mathfrak{t}_m-\mathfrak{t}_s}$. First, we note that the third congruence in (\ref{eqn::system1}) implies that 
$$
\ord_{\pp}(x_2)\geq\min(\mathfrak{t}_m-\mathfrak{t}_{\nu},0)=\mathfrak{t}_m-\min(\mathfrak{t}_m,\mathfrak{t}_{\nu}),
$$
since $x_2\in\oo$. If $\mathfrak{t}_s\geq\mathfrak{t}_{\nu}$, by using assumption (2) and the third congruence in (\ref{eqn::system1}), we have
$$
\ord_{\pp}((s^2+\alpha)x_2^2)\geq2\mathfrak{t}_m+\mathfrak{t}_{\nu}-2\min(\mathfrak{t}_m,\mathfrak{t}_{\nu})\geq\mathfrak{t}_m+e.
$$
Thus (\ref{eqn::preprerelation}) is unsolvable if $x_2\not\in\pp^{\mathfrak{t}_m-\mathfrak{t}_s}$. If $\mathfrak{t}_s<\mathfrak{t}_{\nu}$, by using assumption (3) and again the third congruence in (\ref{eqn::system1}), we have
$$
\ord_{\pp}((s^2+\alpha)x_2^2)\geq\mathfrak{t}_m+\mathfrak{t}_{\nu}-\min(\mathfrak{t}_m,\mathfrak{t}_{\nu})+\ord_{\pp}(x_2)\geq\mathfrak{t}_s+e+1+\ord_{\pp}(x_2)>\ord_{\pp}(2sx_2).
$$
Therefore we see that (\ref{eqn::preprerelation}) is unsolvable if $x_2\not\in\pp^{\mathfrak{t}_m-\mathfrak{t}_s}$. So it suffices to solve the system when $x_2\in\pp^{\mathfrak{t}_m-\mathfrak{t}_s}\cap\pp^{\mathfrak{t}_m-\mathfrak{t}_{\nu}}$.

Suppose that $x_2\in\pp^{\mathfrak{t}_m-\mathfrak{t}_s}\cap\pp^{\mathfrak{t}_m-\mathfrak{t}_{\nu}}$. We count the number of solutions to the system (\ref{eqn::system1}). First, we notice that $sx_2\in\pp^{\mathfrak{t}_m}$. Therefore, we have $x_1\equiv1\pmod{\pp^{\mathfrak{t}_m}}$ by the second congruence in (\ref{eqn::system1}). We write
$$
x_1=1+\pi^{\mathfrak{t}_m}\sum_{i\geq0}^{\infty}\pi^iy_i,\quad z_k\coloneqq1+\pi^{\mathfrak{t}_m}\sum_{i\geq0}^{k-1}\pi^iy_i,
$$
with $y_0,y_1,\ldots\in\oo/\pp$ for $k\geq0$. Using the assumption $\mathfrak{t}_m\geq e+1$, we see that
\begin{equation}
\label{eqn::goingup}
x_1^2\equiv z_k^2+2\pi^{\mathfrak{t}_m+k}y_k\pmod{\pp^{\mathfrak{t}_m+e+k+1}},
\end{equation}
for $k\geq0$. Therefore, the system (\ref{eqn::system1}) is equivalent to the following system
\begin{equation}
\label{eqn::system1alter}
\begin{dcases}
2\pi^{\mathfrak{t}_m}y_0\equiv1-z_0^2-\alpha x_2^2\pmod{\pp^{\mathfrak{t}_m+e+1}},\\
2\pi^{\mathfrak{t}_m+1}y_1\equiv1-z_1^2-\alpha x_2^2\pmod{\pp^{\mathfrak{t}_m+e+2}},\\
\quad\quad\quad\vdots\\
2\pi^{t-1}y_{t-\mathfrak{t}_m-1}\equiv1-z_{t-\mathfrak{t}_m-1}^2-\alpha x_2^2\pmod{\pp^{t+e}},\\
x_2\equiv0\pmod{\pp^{\mathfrak{t}_m-\mathfrak{t}_s}\cap\pp^{\mathfrak{t}_m-\mathfrak{t}_{\nu}}}.
\end{dcases}
\end{equation}
Using the assumption $x_2\in\pp^{\mathfrak{t}_m-\mathfrak{t}_s}\cap\pp^{\mathfrak{t}_m-\mathfrak{t}_{\nu}}$, we see that $1-z_0^2-\alpha x_2^2\in\pp^{\mathfrak{t}_m+e}$. Thus the first congruence in (\ref{eqn::system1alter}) has a unique solution $y_0\in\oo/\pp$. If the first congruence in (\ref{eqn::system1alter}) is consistent, we have $x_1^2+\alpha x_2^2\equiv1\pmod{\pp^{\mathfrak{t}_m+e+1}}$. Combining with (\ref{eqn::goingup}), we see that
$$
1-z_1^2-\alpha x_2^2\equiv1-x_1^2-\alpha x_2^2\equiv0\pmod{\pp^{\mathfrak{t}_m+e+1}}.
$$
Thus, the second congruence in (\ref{eqn::system1alter}) has a unique solution $y_1\in\oo/\pp$. Repeating this argument, we see that $y_0,y_1,\ldots,y_{t-\mathfrak{t}_m-1}$ are uniquely determined by $x_2$. Therefore $x_2$ uniquely determines $x_1$ modulo $\pp^t$. Hence, the number of solutions to the system (\ref{eqn::system1alter}) is the number of elements $x_2\in\oo/\pp^t$ such that $x_2\in\pp^{\mathfrak{t}_m-\mathfrak{t}_s}\cap\pp^{\mathfrak{t}_m-\mathfrak{t}_{\nu}}$, which is $\norm(\pp)^{t-\mathfrak{t}_m+\min(\mathfrak{t}_m,\mathfrak{t}_s,\mathfrak{t}_{\nu})}$.
\end{proof}

\begin{lemma}
\label{thm::system2}
Suppose that $\alpha,\beta\in\oo$ with $1\leq\mathfrak{t}_{\alpha}\coloneqq\ord_{\pp}(\alpha)\leq\min(e,\ord_{\pp}(\beta))$. Assume that $\mathfrak{t}_m\geq e+1$. For $t\geq\mathfrak{t}_m+\mathfrak{t}_{\alpha}$, the number of solutions $(x_1,x_2)$ modulo $\pp^t$ to the system of congruences given by
\begin{equation}
\label{eqn::system2}
\begin{dcases}
x_1^2+\frac{2x_1x_2}{\alpha}+\frac{\beta x_2^2}{\alpha}\equiv1\pmod{\pp^{t+e-\mathfrak{t}_{\alpha}}},\\
x_1\equiv1\pmod{\pp^{\mathfrak{t}_m}},\\
x_2\equiv0\pmod{\pp^{\mathfrak{t}_m}},
\end{dcases}
\end{equation}
is $\norm(\pp)^{t-\mathfrak{t}_m}$.
\end{lemma}
\begin{proof}
Suppose that $(x_1,x_2)$ is a solution to (\ref{eqn::system2}). Combining all congruences in (\ref{eqn::system2}), we have
$$
\frac{2x_1x_2}{\alpha}\equiv0\pmod{\pp^{\mathfrak{t}_m+e}}.
$$
Thus the system (\ref{eqn::system2}) is inconsistent if $x_2\not\in\pp^{\mathfrak{t}_m+\mathfrak{t}_{\alpha}}$. So we assume that $x_2\in\pp^{\mathfrak{t}_m+\mathfrak{t}_{\alpha}}$ henceforth. Since $x_1\equiv1\pmod{\pp^{\mathfrak{t}_m}}$, we write
$$
x_1=1+\pi^{\mathfrak{t}_m}\sum_{i\geq0}^{\infty}\pi^iy_i,\quad z_k\coloneqq1+\pi^{\mathfrak{t}_m}\sum_{i\geq0}^{k-1}\pi^iy_i
$$
with $y_0,y_1,\ldots\in\oo/\pp$ for $k\geq0$. Since $x_2\equiv0\pmod{\pp^{\mathfrak{t}_m}}$, using the assumption $\mathfrak{t}_m\geq e+1$, it is easy to deduce that
$$
x_1^2+\frac{2x_1x_2}{\alpha}\equiv z_k^2+\frac{2z_kx_2}{\alpha}+2\pi^{\mathfrak{t}_m+k}y_k\pmod{\pp^{\mathfrak{t}_m+e+k+1}},
$$
holds for $k\geq0$. Therefore, the system (\ref{eqn::system2}) is equivalent to the following system
\begin{equation}
\label{eqn::system2alter}
\begin{dcases}
2\pi^{\mathfrak{t}_m}y_0\equiv1-z_0^2-\frac{2z_0x_2}{\alpha}-\frac{\beta x_2^2}{\alpha}\pmod{\pp^{\mathfrak{t}_m+e+1}},\\
2\pi^{\mathfrak{t}_m+1}y_1\equiv1-z_1^2-\frac{2z_1x_2}{\alpha}-\frac{\beta x_2^2}{\alpha}\pmod{\pp^{\mathfrak{t}_m+e+2}},\\
\quad\quad\quad\vdots\\
2\pi^{t-\mathfrak{t}_{\alpha}-1}y_{t-\mathfrak{t}_m-\mathfrak{t}_{\alpha}-1}\equiv1-z_{t-\mathfrak{t}_m-\mathfrak{t}_{\alpha}-1}^2-\frac{2z_{t-\mathfrak{t}_m-\mathfrak{t}_{\alpha}-1}x_2}{\alpha}-\frac{\beta x_2^2}{\alpha}\pmod{\pp^{t+e-\mathfrak{t}_{\alpha}}},\\
x_2\equiv0\pmod{\pp^{\mathfrak{t}_m+\mathfrak{t}_{\alpha}}}.
\end{dcases}
\end{equation}
Using the same argument in the proof of Lemma \ref{thm::system1}, one can show that $y_0,y_1,\ldots,y_{t-\mathfrak{t}_m-\mathfrak{t}_{\alpha}-1}$ are uniquely determined by $x_2$. Thus, there are $\norm(\pp)^{\mathfrak{t}_{\alpha}}$ solutions to the system (\ref{eqn::system2alter}) for a fixed $x_2\in\pp^{\mathfrak{t}_m+\mathfrak{t}_{\alpha}}$. Since the number of elements $x_2\in\oo/\pp^t$ such that $x_2\in\pp^{\mathfrak{t}_m+\mathfrak{t}_{\alpha}}$ is $\norm(\pp)^{\mathfrak{t}-\mathfrak{t}_m-\mathfrak{t}_{\alpha}}$, the number of solutions $(x_1,x_2)$ modulo $\pp^t$ is $\norm(\pp)^{t-\mathfrak{t}_m}$.
\end{proof}

\begin{lemma}
\label{thm::system3}
Suppose that $\alpha,\beta,s\in\oo$ such that $1\leq\mathfrak{t}_{\alpha}\coloneqq\ord_{\pp}(\alpha)\leq\min(e,\ord_{\pp}(\beta))$. We set
$$
Q_{\nu}\coloneqq s^2+\frac{2s}{\alpha}+\frac{\beta}{\alpha},\quad D_{\nu}\coloneqq2s+\frac{2}{\alpha},
$$
and set $\mathfrak{t}_{\nu}\coloneqq\ord_{\pp}(Q_{\nu})$. Assume that $\mathfrak{t}_m\geq e+1$. For $t\geq\mathfrak{t}_m+\mathfrak{t}_{\alpha}$, the number of solutions $(x_1,x_2)$ modulo $\pp^t$ to the system of congruences given by
\begin{equation}
\label{eqn::system3}
\begin{dcases}
x_1^2+\frac{2x_1x_2}{\alpha}+\frac{\beta x_2^2}{\alpha}\equiv1\pmod{\pp^{t+e-\mathfrak{t}_{\alpha}}},\\
x_1\equiv1-\left(s+\frac{2}{\alpha}\right)x_2\pmod{\pp^{\mathfrak{t}_m}},\\
x_2\equiv0\pmod{\pp^{\mathfrak{t}_m-\mathfrak{t}_{\nu}}},
\end{dcases}
\end{equation}
is $\norm(\pp)^{t-\mathfrak{t}_m}$.
\end{lemma}
\begin{proof}
Suppose that $(x_1,x_2)$ is a solution to (\ref{eqn::system3}). Using the assumption $\mathfrak{t}_m\geq e+1$, we see that
$$
x_1^2+\frac{2x_1x_2}{\alpha}+\frac{\beta x_2^2}{\alpha}\equiv1+Q_{\nu}x_2^2-D_{\nu}x_2\pmod{\pp^{\mathfrak{t}_m+e}}.
$$
For $x_2\in\pp^{\mathfrak{t}_m-\mathfrak{t}_{\nu}}$, the assumption $\mathfrak{t}_m\geq e+1$ implies that $\ord_{\pp}(Q_{\nu}x_2^2)>\ord_{\pp}(D_{\nu}x_2)$. Therefore the system (\ref{eqn::system3}) is unsolvable if $x_2\not\in\pp^{\mathfrak{t}_m+\mathfrak{t}_{\alpha}}$. If $x_2\in\pp^{\mathfrak{t}_m+\mathfrak{t}_{\alpha}}$, then the second congruence in (\ref{eqn::system3}) becomes $x_1\equiv1\pmod{\pp^{\mathfrak{t}_m}}$ and the third congruence becomes $x_2\equiv0\pmod{\pp^{\mathfrak{t}_m+\mathfrak{t}_{\alpha}}}$. Noting that we have shown that (\ref{eqn::system2}) is unsolvable if $x_2\not\in\pp^{\mathfrak{t}_m+\mathfrak{t}_{\alpha}}$ in the proof of Lemma \ref{thm::system2}, we see that the system (\ref{eqn::system3}) is equivalent to the system (\ref{eqn::system2}). So the number of solutions $(x_1,x_2)$ modulo $\pp^t$ to the system (\ref{eqn::system3}) is $\norm(\pp)^{t-\mathfrak{t}_m}$.
\end{proof}

\section{Local computations}
\label{sec::localcomputation}

In this section, we keep using the local setting in Section \ref{sec::startlocal}. Let $X=L+\nu$ be a shifted lattice in a quadratic space of dimension $2$ over $K$. We are going to compute the group index $[O^{+}(L):O^{+}(X)]$. Since the group index is invariant under scaling, by Lemma \ref{thm::threetypes}, we may assume that $L$ is of one of the following types:
\begin{enumerate}
\item $L\simeq D(1,\alpha)$ with $\alpha\in\oo$.
\item $L\simeq A(0,0)$.
\item $L\simeq A(\alpha,\beta)$ with $\alpha,\beta\in\oo$ such that $1\leq\ord_{\pp}(\alpha)\leq\min(e,\ord_{\pp}(\beta))$.
\end{enumerate}
For the shift $\nu$, we may assume that
$$
\nu=\frac{s_1e_1+s_2e_2}{\pi^{\mathfrak{t}_m}},
$$
with $\mathfrak{t}_m\in\ZZ$ and $s_1,s_2\in\oo$ such that either $s_1\in\oo^{\times}$ or $s_2\in\oo^{\times}$. If $L$ is of type (1), then we set $\mathfrak{t}_{\nu}\coloneqq\ord_{\pp}(s_1^2+\alpha s_2^2)$ and $\mathfrak{t}_s\coloneqq\ord_{\pp}(s_1)$. If $L$ is of type (3), then we set $\mathfrak{t}_{\alpha}\coloneqq\ord_{\pp}(\alpha)$. Furthermore, we set
$$
\mathfrak{t}_L\coloneqq
\begin{dcases}
e,&\text{ if }L\text{ is of type (1)},\\
0,&\text{ if }L\text{ is of type (2)},\\
e-\mathfrak{t}_{\alpha},&\text{ if }L\text{ is of type (3)}.
\end{dcases}
$$
Let $G$ be the Gram matrix of $L$ with respect to the basis so that $L$ is of type (1)-(3). For $t\geq\mathfrak{t}_L+1$, we set
$$
O^{+}(L,\pp^t)\coloneqq\left\{M\in\GL_2(\oo/\pp^t)~\bigg|\begin{array}{c}M^TGM\equiv G\pmod{\pp^{t+\mathfrak{t}_L}}\\\det(M)\equiv1\pmod{\pp^{t+\mathfrak{t}_L}}\end{array}\right\}.
$$
For $t\geq\max(\mathfrak{t}_m,\mathfrak{t}_L+1)$, we set
$$
O^{+}(X,\pp^t)\coloneqq\left\{M\in O^{+}(L,\pp^t)~\bigg|~M\begin{pmatrix}s_1\\s_2\end{pmatrix}\equiv\begin{pmatrix}s_1\\s_2\end{pmatrix}\pmod{\pp^{\mathfrak{t}_m}}\right\}.
$$
For $x_1,x_2\in\oo$, or $\oo/\pp^t$, we define
$$
M_L(x_1,x_2)\coloneqq
\begin{dcases}
\begin{pmatrix}x_1 & -\alpha x_2\\x_2 & x_1\end{pmatrix},&\text{ if }L\text{ is of type (1)},\\
\begin{pmatrix}x_1 & 0\\0 & x_2\end{pmatrix},&\text{ if }L\text{ is of type (2)},\\
\begin{pmatrix}x_1 & -\frac{\beta}{\alpha}x_2\\x_2 & x_1+\frac{2}{\alpha}x_2\end{pmatrix},&\text{ if }L\text{ is of type (3)}.
\end{dcases}
$$

\begin{lemma}
\label{thm::explicitorthogonalgroup}
Any matrix $M\in O^{+}(L)$ is of the form $M=M_L(x_1,x_2)$ with $x_1,x_2\in\oo$ such that $\det(M_L(x_1,x_2))=1$. If $t\geq\max(\mathfrak{t}_m,\mathfrak{t}_L+1)$, then any matrix $M\in O^{+}(L,\pp^t)$ is of the form $M=M_L(x_1,x_2)$ with $x_1,x_2\in\oo/\pp^t$ such that $\det(M_L(x_1,x_2))\equiv1\pmod{\pp^{t+\mathfrak{t}_L}}$.
\end{lemma}
\begin{proof}
One can prove this lemma by solving the equation $M^{T}G=GM^{-1}$ and solving the equation $M^{T}G\equiv GM^{-1}\pmod{\pp^{t+\mathfrak{t}_L}}$. The details are left to the readers.
\end{proof}

\begin{lemma}
\label{thm::rewriteorthogonalgrouptotal}
For $t\geq\max(\mathfrak{t}_m,\mathfrak{t}_L+1)$, there is an isomorphism between groups
$$
O^{+}(L)/O^{+}(X)\simeq O^{+}(L,\pp^t)/O^{+}(X,\pp^t).
$$
\end{lemma}
\begin{proof}
By Lemma \ref{thm::explicitorthogonalgroup} and Hensel's lemma, the reduction map from $O^{+}(L)$ to $O^{+}(L,\pp^t)$ is surjective, which induces a surjective homomorphism from $O^{+}(L)$ to $O^{+}(L,\pp^t)/O^{+}(X,\pp^t)$. Since the kernel of this map is exactly $O^{+}(X)$, we see that $O^{+}(L)/O^{+}(X)\simeq O^{+}(L,\pp^t)/O^{+}(X,\pp^t)$.
\end{proof}

By the isomorphism, it suffices to calculate the group index $[O^{+}(L,\pp^t)\colon O^{+}(X,\pp^t)]$ for sufficiently large $t$. We set
\begin{equation}
\label{eqn::definebeta}
\beta_{\pp}^{+}(X;X)\coloneqq\lim\limits_{t\to\infty}\frac{|O^{+}(X,\pp^t)|}{\norm(\pp)^t}.
\end{equation}
In particular, if $\mathfrak{t}_m=0$, we have
$$
\beta_{\pp}^{+}(X;X)=\beta_{\pp}^{+}(L;L)=\lim\limits_{t\to\infty}\frac{|O^{+}(L,\pp^t)|}{\norm(\pp)^t}.
$$
As long as the limits exist, we can rewrite the group index as
$$
[O^{+}(L,\pp^t)\colon O^{+}(X,\pp^t)]=\frac{\beta_{\pp}^{+}(L;L)}{\beta_{\pp}^{+}(X;X)},
$$
for sufficiently large $t$. In the following theorems, we show that the limits exist and one can explicitly calculate them, providing that $\mathfrak{t}_m$ is sufficiently large.

\begin{theorem}
\label{thm::diagonalp}
Assume that $\pp\nmid2$ and $L$ is of type (1). Then we have
$$
\beta_{\pp}^{+}(X;X)=
\begin{dcases}
1-\eta(-\alpha)\norm(\pp)^{-1},&\text{ if }\mathfrak{t}_m=0,\\
\norm(\pp)^{-\mathfrak{t}_m+\min(\mathfrak{t}_m,\mathfrak{t}_{\nu},\mathfrak{t}_s)},&\text{ if }\mathfrak{t}_m\geq1.
\end{dcases}
$$
\end{theorem}
\begin{proof}
By Lemma \ref{thm::explicitorthogonalgroup}, if $t\geq\max(\mathfrak{t}_m,\mathfrak{t}_L+1)$, then we see that $|O^{+}(X,\pp^t)|$ is the number of the solutions modulo $\pp^t$ to the system
\begin{equation}
\label{eqn::systemp}
\begin{dcases}
x_1^2+\alpha x_2^2\equiv1\pmod{\pp^{t+e}},\\
s_1(x_1-1)-\alpha s_2x_2\equiv0\pmod{\pp^{\mathfrak{t}_m}},\\
s_2(x_1-1)+s_1x_2\equiv0\pmod{\pp^{\mathfrak{t}_m}}.
\end{dcases}
\end{equation}
If $\mathfrak{t}_m=0$, then the formula for $\beta_{\pp}^{+}(X;X)$ follows from Lemma \ref{thm::system0} immediately. It remains to prove the formula for $\beta_{\pp}^{+}(X;X)$ when $\mathfrak{t}_m\geq1$.

First, we assume that $s_2\in\pp$. Then by our assumption $s_1\in\oo^{\times}$ and it follows that $s_1^2+\alpha s_2^2\in\oo^{\times}$. Thus the system (\ref{eqn::systemp}) is equivalent to the system
\begin{equation}
\label{eqn::systempalter1}
\begin{dcases}
x_1^2+\alpha x_2^2\equiv1\pmod{\pp^{t+e}},\\
x_1\equiv1\pmod{\pp^{\mathfrak{t}_m}},\\
x_2\equiv0\pmod{\pp^{\mathfrak{t}_m}}.
\end{dcases}
\end{equation}
Thus, by applying Lemma \ref{thm::system1} with $s=0$ when $\alpha\in\oo^{\times}$ and with $s=s_1$ when $\alpha\in\pp$, we see that the number of solutions to the system (\ref{eqn::systempalter1}) is $\norm(\pp)^{t-\mathfrak{t}_m}$ whenever $t\geq\mathfrak{t}_m$. Therefore, we have
$$
\beta_{\pp}^{+}(X;X)=\norm(\pp)^{-\mathfrak{t}_m}.
$$

Second, we assume that $s_2\in\oo^{\times}$. Then we may assume that $s_2=1$ because multiplying the shift $\nu$ by a unit preserves the group $O^{+}(X,\pp^t)$. In this case, the system (\ref{eqn::systemp}) is equivalent to the system
\begin{equation}
\label{eqn::systempalter2}
\begin{dcases}
x_1^2+\alpha x_2^2\equiv1\pmod{\pp^{t+e}},\\
x_1\equiv1-s_1x_2\pmod{\pp^{\mathfrak{t}_m}},\\
x_2\equiv0\pmod{\pp^{\mathfrak{t}_m-\mathfrak{t}_{\nu}}}.
\end{dcases}
\end{equation}
Therefore, we can apply Lemma \ref{thm::system1} with $s=s_1$ to see that the number of solutions to the system (\ref{eqn::systempalter2}) is $\norm(\pp)^{t-\mathfrak{t}_m+\min(\mathfrak{t}_m,\mathfrak{t}_{\nu},\mathfrak{t}_s)}$ when $t\geq\mathfrak{t}_m$. Hence, we see that
$$
\beta_{\pp}^{+}(X;X)=\norm(\pp)^{-\mathfrak{t}_m+\min(\mathfrak{t}_m,\mathfrak{t}_{\nu},\mathfrak{t}_s)},
$$
as desired.
\end{proof}

If $K$ is dyadic, then $e\geq1$. In order to guarantee that the assumptions of Lemma \ref{thm::system1} hold, we have to assume that $\mathfrak{t}_m$ is sufficiently large.

\begin{theorem}
\label{thm::diagonal2}
Assume that $\pp\mid2$ and $L$ is of type (1). If we assume that $\mathfrak{t}_m\geq\ord_{\pp}(\alpha)+e+1$, then we have
$$
\beta_{\pp}^{+}(X;X)=\norm(\pp)^{-\mathfrak{t}_m+\min(\mathfrak{t}_m,\mathfrak{t}_{\nu},\mathfrak{t}_s)}.
$$
\end{theorem}
\begin{proof}
First, we assume that $s_2\in\pp$. Then, using the arguments in the proof of Theorem \ref{thm::diagonalp}, we have to count the solutions of the system (\ref{eqn::systempalter1}). Since $\mathfrak{t}_m\geq\ord_{\pp}(\alpha)+e+1\geq e+1$, we can invoke Lemma \ref{thm::system1} to see that the number of solutions to the system (\ref{eqn::systempalter1}) is $\norm(\pp)^{t-\mathfrak{t}_m}$ when $t\geq\mathfrak{t}_m$. Therefore, we have
$$
\beta_{\pp}^{+}(X;X)=\norm(\pp)^{-\mathfrak{t}_m}.
$$

Second, we assume that $s_2\in\oo^{\times}$. In this case, we may assume that $s_2=1$. Then using the arguments in the proof of Theorem \ref{thm::diagonalp}, we have to count the solutions of the system (\ref{eqn::systempalter2}). To apply Lemma \ref{thm::system1} with $s=s_1$, we have to show that the following conditions of Lemma \ref{thm::system1} hold:
\begin{enumerate}
\item $\mathfrak{t}_m+\mathfrak{t}_{\nu}-2\min(\mathfrak{t}_m,\mathfrak{t}_{\nu})\geq e$ if $\mathfrak{t}_s\geq\mathfrak{t}_{\nu}$;
\item $\mathfrak{t}_m+\mathfrak{t}_{\nu}-\min(\mathfrak{t}_m,\mathfrak{t}_{\nu})\geq\mathfrak{t}_s+e+1$ if $\mathfrak{t}_s<\mathfrak{t}_{\nu}$.
\end{enumerate}

First, if $\mathfrak{t}_{\nu}=0$, then the condition holds since $\mathfrak{t}_m\geq\ord_{\pp}(\alpha)+e+1$. Second, if $\mathfrak{t}_s\geq\mathfrak{t}_{\nu}\geq1$, then it follows that $\mathfrak{t}_{\nu}=\ord_{\pp}(\alpha)$. Thus we have $\mathfrak{t}_m>\mathfrak{t}_{\nu}$, because otherwise we arrive at a contradiction that $\mathfrak{t}_m\leq\mathfrak{t}_{\nu}=\ord_{\pp}(\alpha)\leq\mathfrak{t}_m-e-1$. Thus, we have
$$
\mathfrak{t}_m+\mathfrak{t}_{\nu}-2\min(\mathfrak{t}_m,\mathfrak{t}_{\nu})=\mathfrak{t}_m-\mathfrak{t}_{\nu}=\mathfrak{t}_m-\ord_{\pp}(\alpha)\geq e.
$$
Third, suppose that $\mathfrak{t}_{\nu}>\mathfrak{t}_s$ and $\mathfrak{t}_{\nu}\geq\mathfrak{t}_m$. Using the assumption $\mathfrak{t}_{\nu}\geq\mathfrak{t}_m\geq\ord_{\pp}(\alpha)+e+1$, we see that $2\mathfrak{t}_s=\ord_{\pp}(\alpha)$. Therefore, we have 
$$
\mathfrak{t}_m+\mathfrak{t}_{\nu}-\min(\mathfrak{t}_m,\mathfrak{t}_{\nu})=\mathfrak{t}_{\nu}\geq\ord_{\pp}(\alpha)+e+1\geq\mathfrak{t}_s+e+1.
$$
Finally, we assume that $\mathfrak{t}_m>\mathfrak{t}_{\nu}>\mathfrak{t}_s$. We claim that $\mathfrak{t}_m\geq\mathfrak{t}_s+e+1$ under this assumption. To see this, if $2\mathfrak{t}_s>\mathfrak{t}_{\nu}$, then $\mathfrak{t}_{\nu}=\ord_{\pp}(\alpha)$. It follows that $\mathfrak{t}_m\geq\mathfrak{t}_{\nu}+e+1>\mathfrak{t}_s+e+1$. If $2\mathfrak{t}_s\leq\mathfrak{t}_{\nu}$, then $2\mathfrak{t}_s\leq\ord_{\pp}(\alpha)$. It follows that $\mathfrak{t}_m\geq2\mathfrak{t}_s+e+1\geq\mathfrak{t}_s+e+1$. Hence, we have
$$
\mathfrak{t}_m+\mathfrak{t}_{\nu}-\min(\mathfrak{t}_m,\mathfrak{t}_{\nu})=\mathfrak{t}_m\geq\mathfrak{t}_s+e+1.
$$
For all of these cases, we have shown that the conditions of Lemma \ref{thm::system1} hold. Therefore, applying Lemma \ref{thm::system1}, we see that
$$
\beta_{\pp}^{+}(X;X)=\norm(\pp)^{-\mathfrak{t}_m+\min(\mathfrak{t}_m,\mathfrak{t}_{\nu},\mathfrak{t}_s)},
$$
as desired.
\end{proof}
\begin{remark}
\label{thm:minimumbound}
It is worth noting that if $L$ is of type (1), then we have
$$
\min(\mathfrak{t}_m,\mathfrak{t}_{\nu},\mathfrak{t}_s)\ll\ord_{\pp}(\alpha).
$$
In fact, if $\ord_{\pp}(\alpha)=0$, then either $\mathfrak{t}_{\nu}=0$ or $\mathfrak{t}_s=0$. Suppose that $\ord_{\pp}(\alpha)\geq1$. If $\mathfrak{t}_m<\ord_{\pp}(\alpha)+e+1$, then the fact is obvious. If $\mathfrak{t}_m\geq\ord_{\pp}(\alpha)+e+1$, then it is a byproduct in the discussion of the four cases in the proof of Theorem \ref{thm::diagonal2}. Since the assumption $\pp\mid2$ is not used in the proof of Theorem \ref{thm::diagonal2}, this fact holds even when $\pp\nmid2$.
\end{remark}

If $L$ is of type (2), then the structure of $O^{+}(X,\pp^t)$ is relatively simple by Lemma \ref{thm::explicitorthogonalgroup}. Therefore, an explicit formula for $\beta_{\pp}^{+}(X;X)$ follows from straightforward calculations.

\begin{theorem}
\label{thm::hyperboliccase}
Assume that $L$ is of type (2). Then we have
$$
\beta_{\pp}^{+}(X;X)=
\begin{dcases}
1-\norm(\pp)^{-1},&\text{ if }\mathfrak{t}_m=0,\\
\norm(\pp)^{-\mathfrak{t}_m},&\text{ if }\mathfrak{t}_m\geq1.
\end{dcases}
$$
\end{theorem}
\begin{proof}
By Lemma \ref{thm::explicitorthogonalgroup}, we see that $|O^{+}(L,\pp^t)|=\norm(\pp)^t-\norm(\pp)^{t-1}$ and $|O^{+}(X,\pp^t)|=\norm(\pp)^{t-\mathfrak{t}_m}$ whenever $t\geq\max(\mathfrak{t}_m,\mathfrak{t}_L+1)$. If $\mathfrak{t}_m=0$, then we have
$$
\beta_{\pp}^{+}(X;X)=1-\norm(\pp)^{-1},
$$
and if $\mathfrak{t}_m\geq1$, then we have
$$
\beta_{\pp}^{+}(X;X)=\norm(\pp)^{-\mathfrak{t}_m},
$$
as desired.
\end{proof}

If $L$ is of type (3), then we have to assume that $\mathfrak{t}_m$ is sufficiently large to apply Lemma \ref{thm::system2} and Lemma \ref{thm::system3}.

\begin{theorem}
\label{thm::remainingcase}
Assume that $L$ is of type (3). If $\mathfrak{t}_m\geq e+1$, then we have
$$
\beta_{\pp}^{+}(X;X)=\norm(\pp)^{-\mathfrak{t}_m}.
$$
\end{theorem}
\begin{proof}
By Lemma \ref{thm::explicitorthogonalgroup}, if $t\geq\max(\mathfrak{t}_m,\mathfrak{t}_L+1)$, we see that $|O^{+}(X,\pp^t)|$ is equal to the number of solutions modulo $\pp^t$ to the system
\begin{equation}
\label{eqn::systemab}
\begin{dcases}
x_1^2+\frac{2x_1x_2}{\alpha}+\frac{\beta x_2^2}{\alpha}\equiv1\pmod{\pp^{t+e-\mathfrak{t}_{\alpha}}},\\
s_1(x_1-1)-\frac{\beta s_2x_2}{\alpha}\equiv0\pmod{\pp^{\mathfrak{t}_m}},\\
s_2(x_1-1)+\left(s_1+\frac{2s_2}{\alpha}\right)x_2\equiv0\pmod{\pp^{\mathfrak{t}_m}}.
\end{dcases}
\end{equation}

First, we assume that $s_2\in\pp$. Then it follows from our assumption that $s_1\in\oo^{\times}$. Thus the system (\ref{eqn::systemab}) is equivalent to the system
\begin{equation}
\label{eqn::systemabalter1}
\begin{dcases}
x_1^2+\frac{2x_1x_2}{\alpha}+\frac{\beta x_2^2}{\alpha}\equiv1\pmod{\pp^{t+e-\mathfrak{t}_{\alpha}}},\\
x_1\equiv1\pmod{\pp^{\mathfrak{t}_m}},\\
x_2\equiv0\pmod{\pp^{\mathfrak{t}_m}}.
\end{dcases}
\end{equation}
Then it follows from Lemma \ref{thm::system2} that the number of solutions modulo $\pp^t$ to the system (\ref{eqn::systemabalter1}) is $\norm(\pp)^{t-\mathfrak{t}_m}$. Hence, we have
$$
\beta_{\pp}^{+}(X;X)=\norm(\pp)^{-\mathfrak{t}_m}.
$$

Second, we assume that $s_2\in\oo^{\times}$. Without loss of generality, we may assume that $s_2=1$. Then the system (\ref{eqn::systemab}) is equivalent to the system
$$
\begin{dcases}
x_1^2+\frac{2x_1x_2}{\alpha}+\frac{\beta x_2^2}{\alpha}\equiv1\pmod{\pp^{t+e-\mathfrak{t}_{\alpha}}},\\
x_1\equiv1-\left(s_1+\frac{2}{\alpha}\right)x_2\pmod{\pp^{\mathfrak{t}_m}},\\
x_2\equiv0\pmod{\pp^{\mathfrak{t}_m-\mathfrak{t}_{\nu}}},
\end{dcases}
$$
with $\mathfrak{t}_{\nu}\coloneqq\ord_{\pp}(s_1^2+\frac{2s_1}{\alpha}+\frac{\beta}{\alpha})$. Applying Lemma \ref{thm::system3} with $s=s_1$, we can conclude that the number of elements in $O^{+}(X,\pp^t)$ is $\norm(\pp)^{t-\mathfrak{t}_m}$. Hence, we have
$$
\beta_{\pp}^{+}(X;X)=\norm(\pp)^{-\mathfrak{t}_m},
$$ 
as desired.
\end{proof}

\section{Proofs of the main results}
\label{sec::main}

Finally, we are ready to prove the main results.

\begin{proof}[Proof of Theorem \ref{thm::maintheorem1}]
By Lemma \ref{thm::rewriteorthogonalgrouptotal} for $\pp\mid\mathfrak{M}_X$, for sufficiently large $t$, we have
$$
[O^{+}(L_{\pp})\colon O^{+}(X_{\pp})]=\frac{|O^{+}(L_{\pp},\pp^t)|}{|O^{+}(X_{\pp},\pp^t)|}=\frac{\beta_{\pp}^{+}(L_{\pp};L_{\pp})}{\beta_{\pp}^{+}(X_{\pp};X_{\pp})}.
$$
Combining with Proposition \ref{thm::classnumberformula}, we have
$$
h^{+}(X)=\frac{h^{+}(L)}{[O^{+}(L):O^{+}(X)]}\prod_{\pp\mid\mathfrak{M}_X}\frac{\beta_{\pp}^{+}(L_{\pp};L_{\pp})}{\beta_{\pp}^{+}(X_{\pp};X_{\pp})}.
$$
By Theorem \ref{thm::diagonalp}, we have
$$
\prod_{\pp\mid\mathfrak{M}_X,\pp\nmid2}\frac{\beta_{\pp}^{+}(L_{\pp};L_{\pp})}{\beta_{\pp}^{+}(X_{\pp};X_{\pp})}=\prod_{\pp\mid\mathfrak{M}_X,\pp\nmid2}\frac{\norm_{K/\QQ}(\pp^{\ord_{\pp}(\mathfrak{M}_X)})}{\norm_{K/\QQ}(\gcd(\pp^{\ord_{\pp}(\mathfrak{M}_X)},\pp^{\ord_{\pp}(\mathfrak{I}_X)})}\prod_{\substack{\pp\mid\mathfrak{M}_X,\pp\nmid2\\M_{\pp}=cD(1,\alpha)}}\left(1-\frac{\eta(-\alpha)}{\norm_{K/\QQ}(\pp)}\right).
$$
By Theorem \ref{thm::diagonal2}, we have
$$
\prod_{\substack{\pp\mid\mathfrak{M}_X,\pp\mid2\\M_{\pp}=cD(1,\alpha)}}\frac{\beta_{\pp}^{+}(L_{\pp};L_{\pp})}{\beta_{\pp}^{+}(X_{\pp};X_{\pp})}=\prod_{\substack{\pp\mid\mathfrak{M}_X,\pp\mid2\\M_{\pp}=cD(1,\alpha)}}\frac{\norm_{K/\QQ}(\pp^{\ord_{\pp}(\mathfrak{M}_X)})\beta_{\pp}^{+}(L_{\pp};L_{\pp})}{\norm_{K/\QQ}(\gcd(\pp^{\ord_{\pp}(\mathfrak{M}_X)},\pp^{\ord_{\pp}(\mathfrak{I}_X)})}.
$$
By Theorem \ref{thm::hyperboliccase} and Theorem \ref{thm::remainingcase}, we have
$$
\prod_{\substack{\pp\mid\mathfrak{M}_X,\pp\mid2\\M_{\pp}\neq cD(1,\alpha)}}\frac{\beta_{\pp}^{+}(L_{\pp};L_{\pp})}{\beta_{\pp}^{+}(X_{\pp};X_{\pp})}=\prod_{\substack{\pp\mid\mathfrak{M}_X,\pp\mid2\\M_{\pp}\neq cD(1,\alpha)}}\norm_{K/\QQ}(\pp^{\ord_{\pp}(\mathfrak{M}_X)})\beta_{\pp}^{+}(L_{\pp};L_{\pp}).
$$
Using the fact that the norm $\norm_{K/\QQ}$ is multiplicative and noting that $\pp\nmid\mathfrak{I}_X$ if $M_{\pp}$ is of type (2)-(3), we can deduce that
$$
h^{+}(X)=\frac{h^{+}(L)}{[O^{+}(L):O^{+}(X)]}\frac{\beta_2^{+}(X)\norm_{K/\QQ}(\mathfrak{M}_X)}{\norm_{K/\QQ}(\gcd(\mathfrak{M}_X,\mathfrak{I}_X))}\prod_{\substack{\pp\mid\mathfrak{M}_X,\pp\nmid2\\M_{\pp}=cD(1,\alpha)}}\left(1-\frac{\eta(-\alpha)}{\norm_{K/\QQ}(\pp)}\right),
$$
as desired.
\end{proof}

\begin{proof}[Proof of Corollary \ref{thm::maintheorem2}]
Without loss of generality, we may assume that $\mathfrak{s}_L\subseteq\oo$. Since $X$ is totally positive, the orthogonal group $O^{+}(L)$ is finite. Therefore, we see that 
$$
[O^{+}(L):O^{+}(X)]\ll_{L}1.
$$
By Remark \ref{thm:minimumbound}, we have
$$
\norm_{K/\QQ}(\gcd(\mathfrak{M}_X,\mathfrak{I}_X))\ll\prod_{\substack{\pp\mid\mathfrak{M}_X\\M_{\pp}=cD(1,\alpha)}}\norm_{K/\QQ}(\pp^{\ord_{\pp}(\alpha)}).
$$
Since $L_{\pp}$ is unimodular for all but finitely many places $\pp\in\Omega_f$, we see that $\norm_{K/\QQ}(\pp^{\ord_{\pp}(\alpha)})=1$ for all but finitely many places $\pp\in\Omega_f$. It follows that
$$
\norm_{K/\QQ}(\gcd(\mathfrak{M}_X,\mathfrak{I}_X))\ll_L1.
$$
Finally, we use a trick of Ramanujan to bound
$$
\prod_{\substack{\pp\mid\mathfrak{M}_X,\pp\nmid2\\M_{\pp}=cD(1,\alpha)}}\left(1-\frac{\eta(-\alpha)}{\norm_{K/\QQ}(\pp)}\right)\gg_{L,\varepsilon}\norm_{K/\QQ}(\mathfrak{M}_X)^{-\varepsilon},
$$
for any positive number $\varepsilon>0$. Fix a place $\pp\mid\mathfrak{M}_X$. If $\eta(-\alpha)<0$, then it is clear that we have
$$
1-\frac{\eta(-\alpha)}{\norm_{K/\QQ}(\pp)}>\norm_{K/\QQ}(\pp^{\ord_{\pp}(\mathfrak{M}_X)})^{-\varepsilon}.
$$
If $\eta(-\alpha)=1$, then the same conclusion holds for all but finitely many places, because the left-hand side is increasing and the right-hand side is decreasing, as $\norm_{K/\QQ}(\pp)$ grows. Since there are only finitely many places with $\eta(-\alpha)=1$ leftover, we see that
$$
1-\frac{\eta(-\alpha)}{\norm_{K/\QQ}(\pp)}\gg_{\epsilon}\norm_{K/\QQ}(\pp^{\ord_{\pp}(\mathfrak{M}_X)})^{-\varepsilon}.
$$
Overall, we see that
$$
\prod_{\substack{\pp\mid\mathfrak{M}_X,\pp\nmid2\\M_{\pp}=cD(1,\alpha)}}\left(1-\frac{\eta(-\alpha)}{\norm_{K/\QQ}(\pp)}\right)\gg_{L,\varepsilon}\norm_{K/\QQ}(\mathfrak{M}_X)^{-\varepsilon}.
$$
Hence, the lower bound follows from Theorem \ref{thm::maintheorem1} immediately.
\end{proof}

\begin{proof}[Proof of Corollary \ref{thm::maintheorem3}]
Since $h^{+}(X)=h$ is fixed, by Corollary \ref{thm::maintheorem2}, the norm of the conductor ideal is bounded. There are only finitely many ideals with a bounded norm. For each fixed ideal $\mathfrak{M}$, there are only finitely many shifted lattices with $\mathfrak{M}_X=\mathfrak{M}$. Hence, there are only finitely many shifted lattices of the form $X=L+\nu$ with $h^{+}(X)=h$.
\end{proof}

\section*{Acknowledgement}

The author would like to thank Zilong He for proofreading Lemma \ref{thm::threetypes}, Ben Kane for his guidance, and Wai Kiu Chan for his useful comments on an earlier version of the results in this paper. The author also wants to express sincere gratitude to anonymous referees for reading this manuscript carefully and for providing insightful comments and suggestions, which significantly improve the clarity and the quality of this paper.

\bibliographystyle{plain}

\end{document}